\documentclass[12pt]{amsart}
\usepackage{geometry} 
\usepackage{amsmath,amsthm,amsfonts,amssymb}
\usepackage{graphics,xy}
\usepackage{url}
\usepackage{xcolor}
\usepackage{graphicx}
\usepackage[capitalise, noabbrev]{cleveref}
\usepackage{ulem}
\usepackage[utf8]{inputenc}
\usepackage{subfigure}

\newtheorem{lemma}{Lemma}[section]
\newtheorem{proposition}[lemma]{Proposition}
\newtheorem{theorem}[lemma]{Theorem}

\theoremstyle{definition}
\newtheorem{definition}[lemma]{Definition}

\newcommand{\Z}{\mathbb{Z}}
\newcommand{\R}{\mathbb{R}}

\newcommand{\C}{\mathbb{C}}

\newcommand{\slC}{\mathrm{SL}_2(\mathbb{C})}
\newcommand{\su}{\mathrm{SU}_2}
\newcommand{\so}{\mathrm{SO}_3}

\DeclareMathOperator{\hol}{hol}
\title[Compatible pants decomposition for surface groups  representations]{Compatible pants decomposition for $\slC$ representations of surface groups}
\author{Renaud Detcherry, Thomas Le Fils and Ramanujan Santharoubane}
\date{} 

\address{Institut de Mathématiques de Bourgogne, UMR 5584 CNRS, Université Bourgogne Franche-Comté, F-2100 Dijon, France}
\email{renaud.detcherry@u-bourgogne.fr}

\address{Institut de Mathématiques de Jussieu , Sorbonne Université, 4 place Jussieu, 75005 Paris}
\email{thomas.le-fils@imj-prg.fr}
\address{Laboratoire de mathématique d’Orsay, UMR 8628 CNRS,
Bâtiment 307, Université Paris-Saclay, 
91405 ORSAY Cedex, FRANCE}
\email{ramanujan.santharoubane@universite-paris-saclay.fr}

\begin{document}

\begin{abstract} For any irreducible representation of a surface group into $\slC$, we show that there exists a pants decomposition where the restriction to any pair of pants is irreducible and where no curve of the decomposition is sent to a trace $\pm 2$ element. We prove a similar property for $\so$-representations. We also investigate the type of pants decomposition that can occur in this setting for a given representation. This result was announced in \cite{DS22}, motivated by the study of the Azumaya locus of the skein algebra of surfaces at roots of unity.
\end{abstract}
\maketitle
\section{Introduction}
\label{sec:intro}Let $\Sigma$ be a compact connected oriented surface (without boundary) of genus at least two. Let $\rho: \pi_1(\Sigma) \longrightarrow G=\slC$ or $\so$ be group homomorphism, we start with the following definition :
\begin{definition}\label{def:compatiblePants}
 A pants decomposition $\mathcal{P}$ of $\Sigma$ is called compatible with $\rho$ if for any curve $c\in \mathcal{P},$ the elements $\pm \rho(c)$ are not unipotent and for any pants $P$ in $\mathcal{P},$ the restriction $\rho|_{\pi_1(P)}$ is irreducible. 
\end{definition}
The purpose of this paper is to prove that any irreducible representation of $\pi_1(\Sigma)$ into $\slC$ or $\so$ admits a compatible pants decomposition. Remark that in the case of a $\slC$ representation, the compatibility condition can be translated into a condition on the traces of the curves of the pants decomposition. The first condition is equivalent to $\mathrm{Tr}(\rho(c))\neq \pm 2,$ and the second is $x^2+y^2+z^2-xyz-4\neq 0,$ where $x,y,z$ are the traces of boundary curves of a pair of pants in the decomposition. One type of pants decomposition important for us is the sausage type which is a pants decomposition in the same orbit, under the action of the mapping class group of $\Sigma$, as the one shown in Figure \ref{fig:saussage}.
\begin{figure}
\includegraphics[scale=0.13]{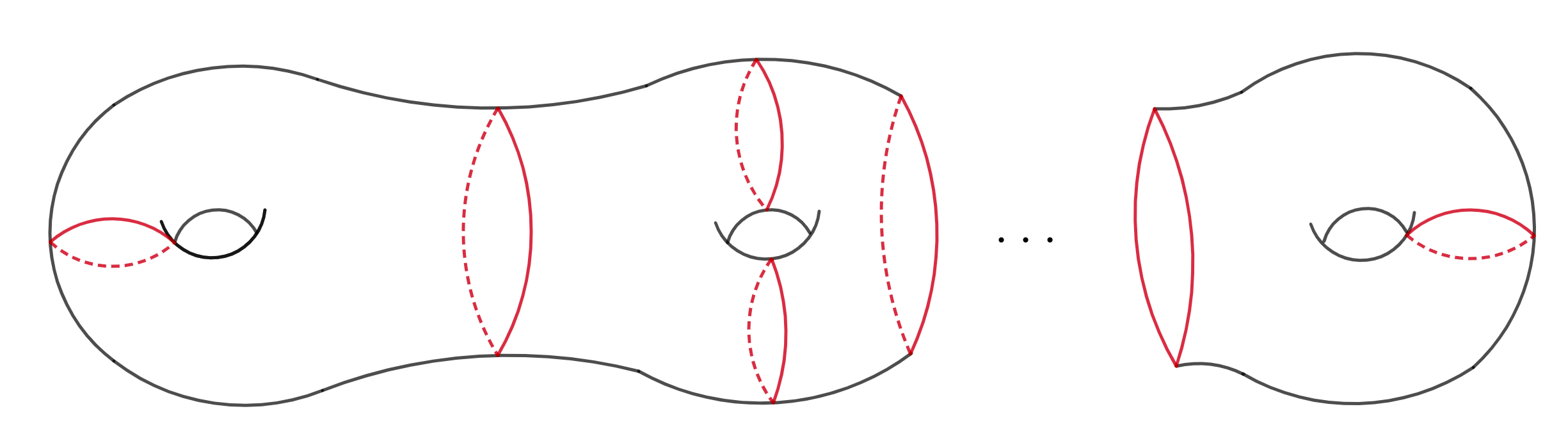}
\caption{A sausage type pants decomposition of $\Sigma$}
\label{fig:saussage}
\end{figure}
\begin{theorem}\label{thm:compatible}
	Let $\rho:\pi_1(\Sigma)\longrightarrow G=\slC$ or $\so$ be an irreducible representation. Then there is a compatible pants decomposition of $\Sigma$ for $\rho.$
	Moreover, if $\rho$ is a representation whose image is not conjugated to the quaternion $Q_8 \subset \su,$ then there is a compatible pants decomposition for $\rho$ of sausage type.
\end{theorem}

The second part of this theorem was announced in \cite[Theorem 1.5]{DS22} without a proof. The sausage type pants decomposition is very important in \cite{DS22} and Theorem \ref{thm:compatible} is key result to understand the Azumaya locus of the skein algebra of $\Sigma$ at roots of unity. Notice that the existence of compatible pants decomposition for non elementary representations in $\slC$ is a key step used by Gallo-Kapovich-Marden in \cite{GKM00} to prove that holonomies of $\C \mathrm{P}^1$-structures are Zariski dense in the $\slC$-character variety of a given surface. It is quite intriguing that exact same condition appeared in \cite{DS22} in the context of quantum topology.
 
A result of Baba \cite{B10} shows that the pants decompositions compatible with a non-elementary representation $\rho$ arising from \cite{GKM00} enables us to construct explicitly all the projective structures with holonomy $\rho$. Our result for representations in $\mathrm{SO}_3$ might also be used to describe the branched spherical structures with given holonomy.

The question of finding compatible pants decomposition also emerged in the first named author's thesis, motivated by Witten's asymptotic expansion conjecture. This conjecture expresses the asymptotics of WRT invariants of a $3$-manifold $M$ as a sum of contributions associated to $\su$ representations of $\pi_1(M)$ and involving Chern-Simons invariants and Reidemeister torsions. The first author conjectures that the geometric quantization techniques from \cite{D18} may be used to estimate the contribution of representations that admit a compatible pants decomposition.

The proof of Theorem \ref{thm:compatible} is split in several steps.
In Section \ref{sec:nonElementary} we deal with representations which are non-elementary, representations with dense images in $\su$ and representations with images in a non-compact dihedral group. For the non elementary case, the existence of compatible pants decomposition is already proved in \cite{GKM00}, we adapt these techniques to get the sausage type pants decomposition. For the case of a representation with dense image in $\su$, Theorem \ref{thm:compatible} is direct application of Previte-Xia's result (see \cite{PX02}) that proved that any such representation has a dense orbit in the $\su$ character variety under the action of the mapping class group. The last case is dealt by hand. In section \ref{sec:finite}, we treat the remaining cases, namely representations with finite images. Such representations are classified, the proof is done by studying the orbits under the mapping class group and building explicit compatible pants decomposition for each orbit. 

\textbf{Acknowledgements:} Over the course of this work, the first named author was supported by the project “AlMaRe” (ANR-19-CE40-0001-01). 
The authors thank Maxime Wolff for very helpful conversations. 

\section{Reduction to representations with finite image}

\label{sec:nonElementary} 
In this section we reduce the proof of \cref{thm:compatible} to the case of representations $\rho$ with finite images. 
More precisely we show \cref{thm:compatible} assuming that \cref{prop:finite} holds. This proposition will be proven in \cref{sec:finite}.

\subsection{Mapping class group action}
Let us begin with some observations on the mapping class group action on sets of representations that will be used throughout the rest of the paper.
Recall that the mapping class group of $\Sigma$ is defined by \[\mathrm{Mod}(\Sigma) = \mathrm{Homeo}^+(\Sigma) / \mathrm{Homeo}^{+}_0(\Sigma).\]
The theorem of Dehn, Nielsen and Baer, see for example \cite[Chapter 8]{FM12}, states that its natural action on the fundamental group induces a group isomorphism $\mathrm{Mod}(\Sigma)\to \mathrm{Out}^+(\Sigma)$ on the index two subgroup $\mathrm{Out}^+(\Sigma)$ of $\mathrm{Out}(\Sigma) = \mathrm{Aut}(\Sigma)/\mathrm{Inn}(\Sigma)$ induced by the automorphisms preserving orientation. Therefore $\mathrm{Mod}(\Sigma)$ acts by precomposition as $\mathrm{Out}^+(\pi_1(\Sigma))$ on the space of conjugacy classes of representations $\mathrm{Hom}(\pi_1(\Sigma), G)/G$, for any group $G$.

A key observation to prove \cref{thm:compatible} is that this action preserves the set of representations admitting a compatible pants decomposition.
If $\rho\colon \pi_1(\Sigma)\to G$ is a representation, we denote by $[\rho]$ its conjugacy class.
\begin{lemma}\label{lem:mcg}
Suppose that $\rho_\infty$ admits a compatible pants decomposition $\mathcal P$. If the conjugacy class $[\rho_\infty]$ of $\rho_\infty$ is in the closure of $\mathrm{Mod}(\Sigma)\cdot [\rho]$, then $\rho$ admits a compatible pants decomposition of the type of $\mathcal P$.
\end{lemma}

\begin{proof}
For a given pants decomposition $\mathrm P$ of $\Sigma$, let us denote by $\mathcal C(\mathrm P)$ the set of conjugacy classes of representations $\pi_1(\Sigma)\to G$ compatible with $\mathrm P$. It follows from the definition of compatibility that these sets are open.
Therefore there exists a representation in $\left ( \mathrm{Mod}(\Sigma) \cdot [\rho]\right )\cap \mathcal C(\mathcal P)$. Hence there exists $f\in \mathrm{Mod}(\Sigma)$ such that $f\cdot [\rho] \in \mathcal C(\mathcal P)$ and therefore $[\rho]\in \mathcal C(f^{-1}\cdot \mathcal P)$.
\end{proof}

We will thus prove \cref{thm:compatible} by studying the orbits of representations $\pi_1(\Sigma)\to \mathrm{SL}_2(\C)$. 
We will use different methods depending on the image of the representation we wish to study.

\subsection{Non-elementary case}
Let us begin with the case where $\rho$ is a non-elementary representation, \textit{i.e.} the action of its image on the Riemann sphere $\mathbb{CP}^1$ by M\"obius transformations has no finite orbit. We can in that case adapt the strategy of Gallo, Kapovich and Marden in \cite[Part A]{GKM00} to find a compatible pants decomposition.
\begin{proposition}\label{prop:nonelementary}
Let $\rho\colon\pi_1(\Sigma)\longrightarrow \mathrm{PSL}_2(\C)$ be a non-elementary representation. For any trivalent graph $\Gamma$ with $3g-3$ edges that has at least one one-edge loop, there is a pants decomposition of $\Sigma,$ with associated graph $\Gamma$ which is compatible with $\rho$.
\end{proposition}
\begin{proof}

Let us begin by recalling the main steps of the construction by Gallo, Kapovich and Marden in \cite{GKM00} of a Schottky pants decomposition for $\rho$: a pants decomposition such that the restriction of $\rho$ to each pair of pants is an isomorphism onto a Schottky group.
Note that this construction of Gallo, Kapovich and Marden works for every non-elementary representation $\pi_1(\Sigma)\longrightarrow \mathrm{PSL}_2(\C)$ except in genus $g=2$ for the pentagon representations. We suppose for now that $\rho$ is not a pentagon representation.

The first step is to find special handle in $\Sigma$. That is a handle $\mathcal H$ whose fundamental group $\pi_1(\mathcal H)$ is sent by $\rho$ onto a non-elementary subgroup of $\mathrm{SL}_2(\C)$. 
This handle $\mathcal H$ allows us to find $g-1$ disjoint simple curves away from it, that are sent by $\rho$ to loxodromic elements.
Cutting the surface along those curves leads to a genus one surface with $2g-2$ boundary components.
Choosing any two of these components, we find a curve separating them from rest of the surface, and such that the restriction of $\rho$ to the pair of pants they bound is an isomorphism onto a Schottky group.
Cutting along this curve takes off a pair of pants and gives a genus one surface with $2g-3$ boundary components.
We repeat the same procedure until we get a genus one surface with two boundary components where a special cutting process is applied to get a Schottky pants decomposition. This process finds a curve bounding the two boundary components and cuts the handle.
We now show that choosing wisely the curves at each step allows us to create a decomposition with any trivalent graph with $3g-3$ edges and at least one one-edge loop. We thus are reduced to the following combinatorial lemma.

\begin{lemma}
Any trivalent graph $\Gamma$ with $3g-3$ edges that has at least one one-edge loop can be created by this procedure.
\end{lemma}
\begin{proof}
We start from a genus one surface $\Sigma$ with $2g-2$ boundary components that come from cutting $g-1$ curves from a closed surface. 
Let us label the boundary components by integers $1\leqslant k\leqslant g-1$ with the same label if they come from cutting the same curve.
Pick a one-edge loop $e$.
Let us consider the graph $\Gamma'$ that is the graph $\Gamma$ with $e$ and all the edges connected to it removed, as for example in \cref{fig:graphe1}.
\begin{figure}[h]
	\scalebox{1.2}{
\begingroup%
  \makeatletter%
  \providecommand\color[2][]{%
    \errmessage{(Inkscape) Color is used for the text in Inkscape, but the package 'color.sty' is not loaded}%
    \renewcommand\color[2][]{}%
  }%
  \providecommand\transparent[1]{%
    \errmessage{(Inkscape) Transparency is used (non-zero) for the text in Inkscape, but the package 'transparent.sty' is not loaded}%
    \renewcommand\transparent[1]{}%
  }%
  \providecommand\rotatebox[2]{#2}%
  \newcommand*\fsize{\dimexpr\f@size pt\relax}%
  \newcommand*\lineheight[1]{\fontsize{\fsize}{#1\fsize}\selectfont}%
  \ifx\svgwidth\undefined%
    \setlength{\unitlength}{198.25041148bp}%
    \ifx\svgscale\undefined%
      \relax%
    \else%
      \setlength{\unitlength}{\unitlength * \real{\svgscale}}%
    \fi%
  \else%
    \setlength{\unitlength}{\svgwidth}%
  \fi%
  \global\let\svgwidth\undefined%
  \global\let\svgscale\undefined%
  \makeatother%
  \begin{picture}(1,0.19460863)%
    \lineheight{1}%
    \setlength\tabcolsep{0pt}%
    \put(0,0){\includegraphics[width=\unitlength,page=1]{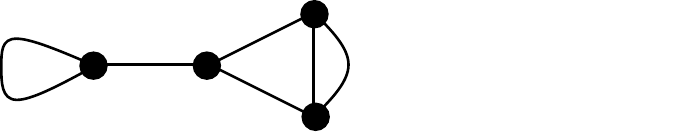}}%
    \put(0.16984097,0.00392075){\makebox(0,0)[lt]{\lineheight{1.25}\smash{\begin{tabular}[t]{l}$\Gamma$\end{tabular}}}}%
    \put(0,0){\includegraphics[width=\unitlength,page=2]{graphe1.pdf}}%
    \put(0.78124279,0.00296539){\makebox(0,0)[lt]{\lineheight{1.25}\smash{\begin{tabular}[t]{l}$\Gamma'$\end{tabular}}}}%
  \end{picture}%
\endgroup%
}
\caption{Example of a trivalent graph $\Gamma$ and $\Gamma'$.}
\label{fig:graphe1}
\end{figure}
Let us cut $\Gamma'$ along $g-1$ edges that do not disconnect it.
We obtain a new graph $\Gamma'$ with $2g-2$ boundary components that we label with integers $1\leqslant k \leqslant g-1$. We require that two boundaries have the same label if they come from the same edge, see the left side of \cref{fig:graphe2}.

The graph $\Gamma'$ has $2g-2$ boundary components and $2g-3$ vertices. 
Therefore one of the vertices has two boundary components, let us choose such a vertex $f$.
Denote by $x$ and $y$ the labels of the boundary components next to $f$. Let us cut the last edge joining $f$ to the rest of the graph and remove the $f$ from $\Gamma'$. We give a new label $z$ to the resulting boundary component, see the right side of \cref{fig:graphe2}.
In the surface $\Sigma$, we pick a curve that that bounds the curves labeled by $x$ and $y$ and take off the pair of pants of pants it defines.
We label by $z$ the new boundary curve. 
\begin{figure}[h]
	\scalebox{1.2}{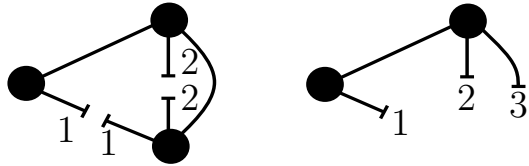}
\caption{The graph $\Gamma'$ in two of the steps.}
\label{fig:graphe2}
\end{figure}
Repeat this procedure until $\Gamma'$ has only one vertex, with two boundary components.
Note that at each step the number of vertices decreases by one, and so does the number of boundary components of $\Gamma'$.
We are left with $\Sigma$ of genus one surface with two boundary components. We then apply the special cutting procedure to finish creating the pants decomposition.
This pants decomposition is isomorphic to $\Gamma$ by construction.
\end{proof}

We now deal with the special case in genus two where $\rho$ is a pentagon representation.
By \cite[Proposition 5.1]{LF19}, we can make the mapping class group act so that the sausage type pants decomposition of \cref{fig:sausage} satisfies:
\begin{itemize}
\item One of the pair of pants is sent by $\rho$ isomorphically onto a Schottky group
\item The other pair of pants is sent by $\rho$ to a non-elementary group and each of its curves is sent by $\rho$ to a loxodromic element.
\end{itemize}
This decomposition is thus compatible with $\rho$.
\end{proof}
We have proven \cref{prop:nonelementary} for representation with values in $\mathrm{PSL}_2(\C)$, which is stronger than the same result for representation in $\mathrm{SL}_2(\C)$. In particular the pentagon representations are not representation $\pi_1(\Sigma)\longrightarrow \mathrm{SL}_2(\C)$: they do not lift do $\mathrm{SL}_2(\C)$.

\subsection{Elementary case}
We now turn to the case where $\rho\colon \pi_1(\Sigma)\longrightarrow \mathrm{PSL}_2(\C)$ is elementary: the action of its image on $\mathbb{CP}^1$ has finite orbits. It is known, see for example \cite[Chapter 5]{R19},  that $\rho$ falls in one of the following three categories:
\begin{enumerate}
\item $\rho$ is affine: it has a conjugate into the upper trianglar matrices.
\item $\rho$ is spherical: it has a conjugate into the group $\mathrm{PSU}_2 = \mathrm{SO}_3$ that preserves the round metric of $ \mathbb{CP}^1 = \mathbb S^2$.
\item $\rho$ is dihedral: it has a conjugate into the group $D$ of matrices that are either diagonal or that have their two diagonal entries vanishing.
\end{enumerate}

Observe that affine representations are reducible. We will therefore only consider representations that either spherical or dihedral.
We begin with the case where $\rho$ is spherical.

\subsubsection{Spherical case}
We now show \cref{thm:compatible} for representation with infinite image in $\mathrm{SO}_3 = \mathrm{PSU}_2$.
\begin{proposition}\label{prop:denseSU2}
Let $\rho\colon\pi_1(\Sigma)\longrightarrow \mathrm{PSU}_2$ be a representation with infinite non-abelian image, there is a pants decomposition of $\Sigma$ of saussage type which is compatible with $\rho.$
\end{proposition}

Let us begin by recalling the description of the closed subgroups of $\mathrm{SU}_2$.
Let us denote by $K$ the subgroup of $\mathrm{SU}_2$ generated by the diagonal matrices and the matrices with both diagonal entries vanishing. In other words, $K = \mathrm{SU}_2\cap D$.

\begin{lemma}
Let $H$ be a subgroup $\su$. One of the following holds:
\begin{enumerate}
\item $H$ is finite
\item $H$ has a conjugate into the diagonal matrices
\item $H$ has a conjugate dense in $K$
\item $H$ is dense in $\mathrm{SU}_2$.
\end{enumerate}
\end{lemma}

\begin{proof}
Replacing $H$ by its closure if necessary, we may assume that $H$ is closed and is a Lie subgroup of $\mathrm{SU}_2$. 
Let us recall that the Lie alebra $\mathfrak {su}(2)$ is isomorphic to $\mathbb R^3$ endowed with the cross-product. 
This Lie algebra does not admit any two dimensional subalgebra. Hence the Lie algebra $\mathfrak h$ of $H$ must be of dimension $0$, $1$ or $3$.
\begin{itemize}
\item If $\dim \mathfrak h = 0$ then $H$ is a discrete group and since $\mathrm{SU}_2$ is compact $H$ is finite. 
\item If $\dim \mathfrak h = 3$ then $H$ is the whole connected group $\mathrm{SU}_2$.
\item  If $\dim \mathfrak h = 1$ then let us conjugate $H$ to that its connected component containing the identity is 
$$\mathrm{Diag} =
\left \{
\begin{pmatrix}
e^{it} & 0\\
0 & e^{-it}
\end{pmatrix}
\mid t\in \mathbb R\right \}
.$$
Let $h\in H$. The group $h\mathrm{Diag}h^{-1}$ is connected and contains the identity hence is included in $H$. Therefore $h$ must send the eigenvectors of matrices of $\mathrm{Diag}$ to themselfs. Let us denote by $(e_1, e_2)$ the canonical basis of $\mathbb C^2$. The vector $he_1$ is an eigenvector of a matrix in $\mathrm {Diag}$ and thus is in either $\mathbb R e_1$ or $\mathbb R e_2$.
Hence $h$ is either diagonal or has the form 
$\begin{pmatrix}
0 & -e^{-i\theta}\\
e^{i\theta} & 0
\end{pmatrix}.
\qedhere$
\end{itemize}
\end{proof}

It follows that the infinite subgroups of $\mathrm{SO}_3(\R) = \mathrm{PSU}_2$ are, after conjugation either abelian, dense in $\mathrm{O}_2(\R)$ or dense in $\mathrm{SO}_3(\R)$.
The abelian representations are reducible, therefore we will only consider the two other cases.
To prove \cref{prop:denseSU2} it suffices to show that in the closure of the orbit of such a $\rho$, there exists a representation with finite non-abelian image. Indeed by \cref{lem:mcg} and \cref{prop:finite} the representation $\rho$ then admits a compatible pants decomposition of sausage type.
\cref{prop:denseSU2} follows from the following lemma.
\begin{lemma}\label{lem:so3}
Every representation $\rho\colon \pi_1(\Sigma)\longrightarrow \mathrm{PSU}_2$ with infinite non-abelian image admits a representation with finite non-abelian image in the closure of its $\mathrm{Mod}(\Sigma)$-orbit.
\end{lemma}

We will actually show that for representations in $\mathrm{SO}_3$ with dense image, we can obtain any trivalent graph $\Gamma$.
\begin{lemma}
For every trivalant graph $\Gamma$ with $3g-3$ edges and every representation $\rho\colon \pi_1(\Sigma)\longrightarrow \mathrm{SO}_3$ with dense image, there exists a compatible pants decomposition whose graph is $\Gamma$.
\end{lemma}
\begin{proof}
The mapping class group orbit of $\rho$ is dense in its connected component of $\mathrm{Hom}(\pi_1(\Sigma), \mathrm{SO}_3(\R))/\mathrm{SO}_3(\R).$ This is a consequence of the main result of \cite{PX02}, see also \cite[Lemma 4.4]{LF21}. Let us recall that this connected component is the set of all representations $\pi_1(\Sigma)\longrightarrow \mathrm{PSU}_2 = \mathrm{SO}_3(\R)$ that lifts to $\su$ if $\rho$ does (resp. that do not lift if $\rho$ does not). Each of these components admit at least one representation with a compatible pants decomposition of graph $\Gamma$, see \cite[Section 6]{M12} and \cite{JW94}. The result follows from \cref{lem:mcg}.
\end{proof}

We now prove \cref{lem:so3} for the representation $\pi_1(\Sigma)\longrightarrow \mathrm{SO}_3$ whose image is not dense in $\mathrm{SO}_3$.

\begin{proof}
We are left with the case where $\rho$ has its image dense in $\mathrm{O}_2(\R)$. Its $\mathrm{Mod}(\Sigma)$-orbit is dense in the representations $\pi_1(\Sigma)\longrightarrow \mathrm{O}_2(\R)$ that are not abelian and lift to $\su$ if $\rho$ does (resp. do not lift if $\rho$ does not), see \cite[Proposition 4.2]{LF21}.
These sets contain representations with finite non-abelian image. For example we can take a representation with image a finite Dihedral group.
\end{proof}

\subsubsection{Dihedral case}

We now turn to the case where $\rho$ is dihedral, that is when it has image in $D$.
The group $D$ is naturally isomorphic to $\mathbb C^\star \rtimes \mathbb Z_2$. 
Let us denote by $\epsilon\colon \mathbb C^\star \rtimes \mathbb Z_2\longrightarrow \mathbb Z_2$ the projection on the second factor.

\begin{proposition}
Let $\rho\colon\pi_1(\Sigma)\longrightarrow H$ be an irreducible representation with infinite non-abelian image. There exists a decomposition of sausage type which is compatible with $\rho$.
\end{proposition}

Fix $a_1, \ldots, b_g$ a standard generating set for $\pi_1(\Sigma)$.
Let us observe that the curves freely homotopic to $a_i$ for $1\leqslant i \leqslant g$ and the ones freely homotopic to $\prod_{i=k}^g \left [a_i, b_i\right ]$ can be completed in a sausage type pants decomposition, see \cref{fig:genus5vert}.

\begin{figure}[h]
\includegraphics[scale=0.45]{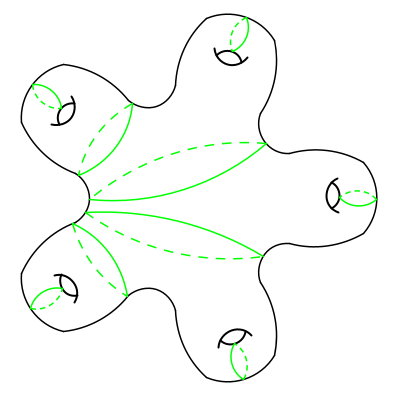}
\caption{Part of a sausage type pants decomposition.}
\label{fig:genus5vert}
\end{figure}

\begin{proof}We have proven this result when $\rho$ takes its values in $\mathrm{SU}_2$. Therefore we now assume that $\rho$ does not admit a conjugate in $\mathrm{SU}_2$. 

It suffices to prove that we can assume that the two following conditions are met, after precomposing $\rho$ with an automorphism of $\pi_1(\Sigma)$:
\begin{enumerate}
\item $\epsilon\circ \rho(a_i) = 1$ for all $1\leqslant i\leqslant g$
\item $\rho\left (\prod_{i=k}^g \left [a_i, b_i\right ]\right )$ is loxodromic for $1 < k\leqslant g$.
\end{enumerate}
Indeed we have seen that the curve freely homotopic to these loops form part of a pants decomposition of sausage type. If $\rho$ satisfies these two conditions, then the restriction of $\rho$ to each pair of pants is not reducible: the curve homotopic to $\prod_{i=k}^g \left [a_i, b_i\right ]$ preserves only the lines generated by the vectors of the standard basis of $\mathbb C^2$, while the one homotopic to $a_i$ does not preserve any of them. Moreover, two of the curves of each pants have zero entries on the diagonal and hence are not central. The third one is non-central as well because it is loxodromic.

We now prove that we may assume that those two conditions are met, with an action of the mapping class group.

The mapping class group  acts transitively on the set of epimorphisms $\pi_1(\Sigma)\longrightarrow \mathbb Z_2$, therefore as before we may assume that $\epsilon\circ \rho(\gamma) = 0$ for all $\gamma \in \{a_1, b_1, \ldots, b_g\}\setminus \{b_1\}$ and $\epsilon\circ \rho(\gamma) = 1$ for $\gamma = b_1$. As before we can conjute $\rho$ by a diagonal matrix so that $\rho(b_1) = \begin{pmatrix}
0 & 1\\
-1 & 0
\end{pmatrix}$. The matrices $\rho(a_1)$ and $\rho(b_1)$ must commute, hence $\rho(a_1) \in \left \{ \begin{pmatrix}
1 & 0\\
0 & 1
\end{pmatrix}, 
-\begin{pmatrix}
1 & 0\\
0 & 1
\end{pmatrix}\right \}$.
Note that the restriction $\rho_{\vert \Sigma_{g-1}}$ of $\rho$ to the subsurface obtained by removing the handle contaning $a_1$ and $b_1$ takes its values in $\ker \epsilon \simeq \mathbb C^\star$.
Let $f = \log |\rho_{\Sigma_{g-1}}|\colon \pi_1(\Sigma)\longrightarrow \mathbb R$. 
This homomorphism is not trivial since $\rho$ does not take its values in $\mathrm{SU}_2$. 
We may act by the mapping class group of $\Sigma_{g-1}$ to make $f(a_g)$ non-zero: we may exchange the handles so that either $f(a_g)$ or $f(b_g)$ does not vanish. Then applying a Dehn twist along $b_g$ if necessary we may assume that $f(a_g)\neq 0$.
Applying a power Dehn twist along $a_g$ we may assume that $f(b_g) + \sum_{i=1}^{g-1} f(b_i) > 1$.

Let $\gamma = b_1\ldots b_{g-1} b_g$. Let us apply a Dehn twist along a curve freely homotopic to $\gamma$. It leaves unchanged each $b_i$ and changes each $a_i$ so that $\epsilon\circ \rho(a_i) = 1$.
Moreover \[\prod_{i=k}^g \left [\rho(a_i), \rho(b_i)\right ] = \prod_{i=k}^g \left [\begin{pmatrix}
0 & \mu_i\\
-\mu_i^{-1} & 0
\end{pmatrix}, \begin{pmatrix}
\lambda_i & 0\\
0 & \lambda_i^{-1}
\end{pmatrix}\right ] = \prod_{i=k}^g \left [ \begin{pmatrix}
\lambda_i^{-2} & 0\\
0 & \lambda_i^2
\end{pmatrix}
\right ].\]

Therefore we have $\rho\left (\prod_{i=k}^g \left [a_i, b_i\right ]\right ) = \begin{pmatrix}
\prod_{i\geqslant k} \lambda_i^{-2} & 0 \\
0 & \prod_{i\geqslant k} \lambda_i^{2}
\end{pmatrix}$ that is loxodromic for every $1 < k \leqslant g$.
\end{proof}

\section{Representations onto finite subgroups of $\su$ or $\so$}
\label{sec:finite}
We recall that pants decomposition of sausage were introduced in Figure \ref{fig:saussage}. We will also consider pants decomposition of \textit{square type}, to be the ones whose dual graph is represented in Figure \ref{fig:xyz}.

This section is devoted to the proof of the following proposition:
\begin{proposition}\label{prop:finite}
	Let $\Sigma$ be a closed compact oriented surface.
	\begin{itemize}
		\item[-] 
		Let $G$ be a nonabelian finite subgroup of $\so,$ and $\rho:\pi_1(\Sigma)\longrightarrow G$ be an epimorphism. Then there is a pants decomposition of $\Sigma$ of sausage type,  and there is one of square type, which are compatible with $\rho.$
		\item[-] 
		Let $\rho:\pi_1(\Sigma) \longrightarrow Q_8$ be an epimorphism. Then there is pants decomposition of $\Sigma$ of square type which is compatible with $\rho.$
	\end{itemize}
\end{proposition}
To exhibit such appropriate epimorphisms into finite groups, we will define them using cocycles on a cellular decomposition of the surface, which we define in Section \ref{sec:cocycle}.. 
\subsection{Representations as holonomy of cocycles}
\label{sec:cocycle}

For $X$ a CW-complex, we write $\mathcal{C}^i(X)$ for the set of its oriented $i$-dimensional cells.
\begin{definition}\label{def:cocycle}
	Let $X$ be a CW-complex and $G$ a group. A $G$-cocycle on $X$ is a map $c: \mathcal{C}^1(X) \longrightarrow G$ that satisfies the following properties:
	\begin{itemize}
		\item[-] For any oriented edge $e \in \mathcal{C}^1(X),$ we have $c(\overline{e})=c(e)^{-1}$ where $\overline{e}$ is the edge $e$ with opposite orientation.
		\item[-] For any $2$-cell $w\in \mathcal{C}^2(X)$ with boundary $e_1  e_2 \ldots e_k,$ we have $$c(e_1)c(e_2)\ldots c(e_k)=1_G.$$
	\end{itemize} 
\end{definition}
For $G$ a group, $G$-cocycles on a CW complex $X$ correspond to the "local" version of representations of $\pi_1(X)$ into $G.$ Indeed, one recovers representations of $\pi_1(X)$ taking the holonomy:
\begin{lemma}\label{lemma:cocycle}
	For $c$ a $G$-cocycle on a CW-complex $X,$ for $x_0\in \mathcal{C}^0(X)$ and $\gamma=e_1\ldots e_n$ a loop on the $1$-skeleton of $X$ based at $x_0,$ let us write
	$$\hol_c(\gamma)=c(e_1)\ldots c(e_n).$$
	Then:
	\begin{itemize}
		\item[-] For $c$ a $G$-cocycle, the map $\gamma \mapsto \hol_c(\gamma)$ is a morphism $\pi_1(X,x_0)\longrightarrow G.$
		\item[-] For any morphism $\varphi :\pi_1(X,x_0)\longrightarrow G,$ there is a $G$-cocycle on $X$ such that $\varphi=\hol_c.$
	\end{itemize}
\end{lemma}
\begin{proof}
	Any loop in $X$ can be homotoped to lie in the $1$-skeletton of $X,$ and any homotopy between loops in the $1$-skeleton can be homotoped to lie in the $2$-skeleton. Moreover, the second condition of Definition \ref{def:cocycle} ensures that the holonomy of a loop depends only on its homotopy class.
	
	For the second point, let us consider a covering tree $T$ in the $1$-skeleton and let $\tilde{X}=X/\langle T\rangle$ be the CW-complex obtained by contracting $T$ to a point. Then since $\tilde{X}$ has only one $0$-cell, a choice of $G$-cocycle on $\tilde{X}$ is equivalent to a choice of representation $\rho:\pi_1(\tilde{X})\longrightarrow G.$ Note that $\tilde{X}$ has the same $\pi_1$ as $X,$ and extending a $G$-cocycle on $\tilde{X}$ by setting $c(e)=1_G$ for any $e\in T,$ one gets a $G$-cocycle whose holonomy is $\rho.$
\end{proof}
$G$-cocycles with the same holonomy can be related thanks to the following proposition:
\begin{lemma}
	Let $X$ be a CW-complex, $G$ a group, $x_0\in \mathcal{C}^0(X)$ and let $c$ and $c'$ be two $G$-cocycle with the same holonomy. Then there exists a map $d:C^0(X)\longrightarrow G$ such that $d(x_0)=1_G$ and for any oriented edge $e$ with $\partial e=x \cup \overline{y}$ we have
	$$c'(e)=d(x)c(e)d(y)^{-1}.$$
\end{lemma}
\begin{proof}
	Choose a maximal covering tree out of the $1$-skeleton of $X.$ We can pick the value of $d$ one vertex of $C^0(X)$ at a time, so that the relation $c'(e)=d(x)c(e)d(y)^{-1}$ is satisfied for any edge $e$ belonging to $T.$ One can further assume that $d(x_0)=1_G,$ by picking $x_0$ as the first vertex.

	Now define $c''$ by $c''(e)=d(x)c(e)d(y)^{-1}$ for any edge $e$ of $C^1(X).$ It is clear that $c''$ is also a $1$-cocycle, that has the same holonomy as $c$ or $c',$ and $c''$ coincides with $c'$ on $T.$ We claim that it coincides with $c'$ on the remaining edges. Indeed, for any edge $e$ not in $T,$ there is a loop on $C^1(X)$ based at $x_0$ whose only edge not in $T$ is $e.$ Then the fact that $c'$ and $c''$ have the same holonomy and coincide on edges in $T$ implies that they coincide on $e$ too.
\end{proof}
Now let $X=\Sigma_g$ with some fixed cellular decomposition, we will need a criterion for when a representation $\rho=\hol_c$ of $\pi_1(\Sigma_g)$ into $\so$ lifts to a representation $\tilde{\rho}:\pi_1(X)\longrightarrow \su.$ Since any representation of $\pi_1(\Sigma_g)$ into $\su$ is represented by a $\su$-cocycle, $\rho$ will admit a lift if and only if $c$ lifts to a $\su$-cocycle $\tilde{c}.$

For any oriented edge $e$ of the $1$-skeleton of $\Sigma_g,$ choose a lift $\tilde{c}(e)\in \su$ of $c(e)\in \so.$ The lifts can be chosen so that $\tilde{c}(\overline{e})=\tilde{c}(e)^{-1},$ are unique up to multiplication by $\pm I_2.$ Let us fix a choice of lift of the $\so$-cocycle $c$ to a map $\tilde{c}:\mathcal{C}^1(\Sigma_g)\longrightarrow \su,$ and an orientation of $\Sigma_g.$ For $f\in \mathcal{C}^{2,+}(\Sigma_g),$ the set of oriented $2$-cells whose orientations agree with that of $\Sigma_g,$ let $\partial f=e_1\ldots e_k$ be the oriented boundary. We define:
$$\varepsilon(c)=\underset{f\in \mathcal{C}^{2,+}(\Sigma_g)}{\prod}\hol_{\tilde{c}}(\partial f)\in \su$$
Note that since $c$ is a $\so$-cocycle, we have that $\hol_c(\partial f)=1_{\so}$ for any $f\in \mathcal{C}^{2,+}(\Sigma_g),$ hence $\varepsilon(c)\in \lbrace \pm I_2 \rbrace.$ Moreover, since each edge appears twice in the product, and the choice of lifts $\tilde{c}(e)$ are unique up to $\pm I_2,$ the quantity $\varepsilon(c)$ is independant of the choice of lift $\tilde{c}.$
\begin{lemma}\label{lemma:lift} Let $\rho=hol_c$ be a representation of $\pi_1(\Sigma_g,x_0)$ into $\so.$ Then $\rho$ lifts to a $\su$ representation of $\pi_1(\Sigma_g,x_0)$ if and only if $\varepsilon(c)=I_2.$	
\end{lemma}
\begin{proof}
	Note that if $c$ lifts to a $\su$ cocycle $\tilde{c},$ then $\varepsilon(c)=I_2,$ since $\varepsilon(c)$ does not depend on the choice of a lift $\tilde{c}:\mathcal{C}^1(\Sigma_g)\longrightarrow \su,$ we must have $\varepsilon(c)=I_2.$ Moreover we claim that if the representation $\rho=\hol_c$ lifts, then so does the cocycle $c.$ Indeed, the lift $\tilde{\rho}$ will be the holonomy of a $\su$ cocycle $c'.$ We may not have that $r(c')=c,$ however, $r(c')$ and $c$ have the same holonomy, hence they are related by a map $d:\mathcal{C}^0(\Sigma_g)\longrightarrow \so,$ so that $r(c')(e)=d(a)c(e)d(b)^{-1}$ if $\partial e=a \cup \overline{b}.$ Since there is no obstruction to lift the map $d$ to $\su,$ after correcting $c'$ by a lift $d'$ of $d,$ we get a lift of the cocycle $c$ to $\su.$

\end{proof}
\subsection{Compatible pants for representations with finite image}
We now want to produce for an irreducible representation of $\pi_1(\Sigma_g)$ in $\so$ or $\su$ with finite image a compatible pants decomposition. We note that for the groups $\so$ and $\su$ being irreducible is equivalent to having nonabelian image. Let us write $r:\su \longrightarrow \so\simeq \su/\lbrace \pm I_2 \rbrace$ for the projection map. When  $\rho:\pi_1(\Sigma_g)\longrightarrow \su$ is irreducible, it can happen that the composition $r\circ \rho$ is also irreducible; in that case, a compatible pants decomposition for $r \circ \rho$ will also be one for $\rho.$ This will allow us to focus on the case of representations into $\so,$ with the exception of representations with image the quaternion group $Q_8,$ as the following lemma shows: 
\begin{lemma}
	Let $G$ be a nonabelian finite subgroup of $\su$ such that $r(G)< \so$ is abelian. Then $G$ is conjugated in $\su$ to $Q_8.$ 
\end{lemma}
\begin{proof}
	First we claim that $-I_2$ must be in $G,$ otherwise $G$ and $r(G)$ would be isomorphic. Then $r(G)\simeq G/\lbrace \pm I_2 \rbrace$ is the quotient of $G$ by a subgroup of its center. The subgroup $r(G)< \so$ being abelian, it is either cyclic or (up to conjugation) the subgroup $D < \so$ of diagonal matrices with $\pm 1$ diagonal coefficients and determinant $1.$ However, if $r(G)$ was cyclic, this would imply that $G$ itself is abelian. Hence we are in the second case  which proves the claim since $r^{-1}(D)=Q_8.$
\end{proof}
The next lemma is standard, a proof may be for example found in \cite{artin_2018}:
\begin{lemma}
	Nonabelian finite subgroups of $\so$ are isomorphic to either a dihedral group $D_n$ with $n\geq 3,$ to the symetric group $S_4,$ or the alternating group $A_4$ or $A_5.$
\end{lemma}
\begin{proposition}\label{prop:orbit_finite}\cite[Proposition 1.8]{LF21}
	Let $g\geq 2,$ let $G=D_n$ with $n\geq 3$ or $S_4, A_4$ or $A_5.$
	For $\varepsilon \in \lbrace +,-\rbrace,$ let $\mathrm{Hom}^{s,\varepsilon}(\pi_1(\Sigma_g),G)$ be the set surjective morphisms $\pi_1(\Sigma_g)\longrightarrow G$ that lift (resp. do not lift) to $\su$ if $\varepsilon=+$ (resp. $\varepsilon=-$). Then $\mathrm{Mod}(\Sigma_g)$ acts transitively on $\mathrm{Hom}^{s,\varepsilon}(\pi_1(\Sigma_g),G).$ 
\end{proposition}
We note that for $G=D_n$ with $n$ odd, any surjective morphism lifts to $\su.$ We need to complement this proposition with the case of surjective morphisms onto the quaternion group:
\begin{proposition}\label{prop:orbit_quaternion}
Let $g\geq 2,$ the mapping class group $\mathrm{Mod}(\Sigma_g)$ acts transitively on $\mathrm{Hom}^{s}(\pi_1(\Sigma_g),Q_8)$ 
\end{proposition}
\begin{proof}
Let $a_1,b_1,\ldots, a_g,b_g$ be standard generators of $\pi_1(\Sigma_g).$ We recall that the quaternion group has generators and relations: $Q_8=\langle I,J,K,-1 |I^2=J^2=K^2=-1, (-1)^2=1, IJ=K,KI=J, JK=I \rangle.$ Following the strategy of the proof of Proposition \ref{prop:orbit_finite}, we will deduce the proposition from the fact that $\mathrm{Mod}(\Sigma_g)$ acts transitively on surjective morphisms onto a cyclic group \cite{Edmonds} (see also \cite[Theorem 3.3]{LF21} for a short proof). Let 
$\rho:\pi_1(\Sigma_g)\longrightarrow Q_8$ be a surjective morphism, then the quotient morphism onto $Q_8/\lbrace J \rbrace\simeq \Z/2\Z$ is surjective. Up to mapping class group action we can assume that $\rho(a_1)\in I\langle J\rangle$ and that $\rho(b_1),\rho(a_i), \rho(b_i)	\in \lbrace J \rbrace$ for $i\geq 2.$ We claim that actually $\rho(b_1)=\pm 1,$ indeed, otherwise one would have $[\rho(a_1),\rho(b_1)]=-1,$ while $[\rho(a_i),\rho(b_i)]=1$ for $i\geq 2.$ Therefore $\rho$ would not satisfy the surface group relation. 

Moreover, up to applying $t_{a_1}^2$ where $t_{a_1}$ is the Dehn twist along $a_1,$ one may assume that $\rho(b_1)=1.$ (as $t_{a_1}^2(b_1)=b_1a_1^2.$) 

Now since $\rho(a_1)\in I\langle J \rangle$ $\rho(b_1)=1$ and $\rho$ is surjective onto $Q_8,$ one must have that $\rho|_{\langle a_2,b_2,\ldots,a_g,b_g\rangle}$ is surjective onto $\langle J \rangle\simeq \Z/4\Z.$ 
Without loss of generality, assume that $a_2$ is mapped to a generator of $\langle J \rangle.$ The loop $b_1a_2$ is representated by a simple closed curve disjoint from $b_1,$ and one has $t_{b_1 a_2}^k(a_1)=a_1(b_1 a_2)^k.$ Hence up to applying a power of this Dehn twist, one may assume that $\rho(a_1)=I.$
Finally, up to mapping class group action of the subsurface with fundamental group $\langle a_2,b_2,\ldots a_g,b_g \rangle,$ one may assume that $\rho(a_1)=I, \rho(b_1)=1, \rho(a_2)=J,\rho(b_2)=1,$ and $\rho(a_i)=\rho(b_i)=1$ for $i\geq 3.$
\end{proof}
We recall that the orbits of the action of $\mathrm{Mod}(\Sigma_g)$ on pants decomposition of $\Sigma_g$ correspond to the isomorphism classes of the dual trivalent graphs of pants decompositions.
Thanks to Proposition \ref{prop:orbit_finite} and \ref{prop:orbit_quaternion}, to show that every surjective morphisms of $\pi_1(\Sigma_g)$ onto a finite nonabelian group $G$ of $\su$ or $\so$ admits a compatible pants decomposition of some fixed type, we only need to exhibit one such surjective morphism which is compatible to a fixed pants decomposition of the same type. In addition, for $G=D_n$ with $n$ even, $S_4,$ $A_4$ or $A_5,$ we need to find one such surjective morphism that lifts to $\su$ and one that does not.

We will exhibit such surjective morphism for two types of pants decomposition: one with a sausage dual graph, and one with "the square chain" dual graph. They represented in Figure \ref{fig:xyz} and \ref{fig:sausage}. Note that the pants in the pants decomposition can be further cut by three arcs into a pair of hexagons, so that $\Sigma_g$ is obtained by gluing two copies of (a banded version of) its dual graph, and so that we get a cellular decomposition of $\Sigma_g,$ with $2$-cells that are hexagons. This is represented in Figure \ref{fig:xyz} and \ref{fig:sausage}.

We define a $G$-cocycle for these cellular decomposition by specifying the holonomy of each arc, and checking that the holonomy around each hexagon is $1.$
\begin{figure}
\includegraphics[scale=0.35]{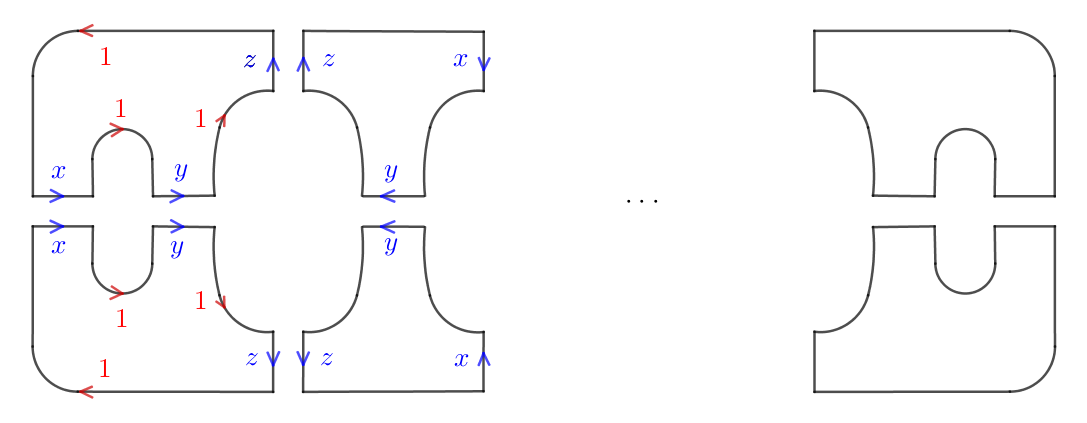}
\caption{The top half of a cellular decomposition of $\Sigma_g$ associated to the square chain dual graph, and an associated cocycle $c.$ The cocycle $c$ has value $1$ on all edges of the lower copy of the dual graph.}
\label{fig:xyz}
\end{figure}
\begin{figure}
\includegraphics[scale=0.35]{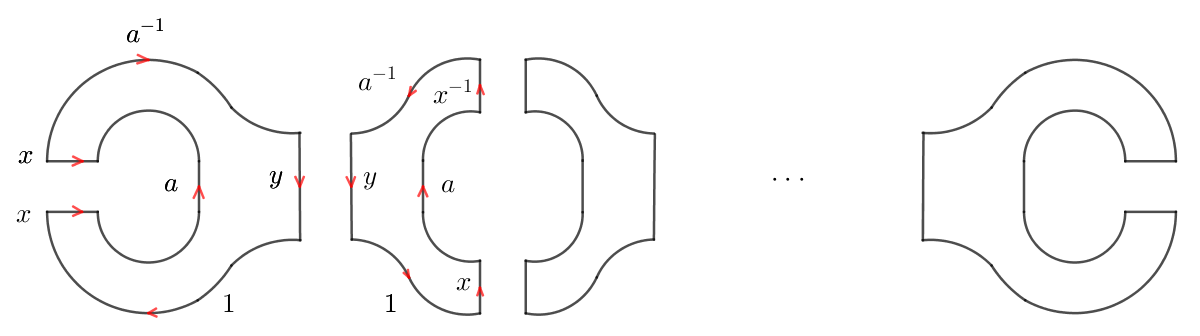}
\caption{The top half of a cellular decomposition of $\Sigma_g$ associated to the sausage dual graph, and an associated cocycle $c.$ The cocycle $c$ has value $1$ on all remaining edges of the lower copy of the dual graph.}
\label{fig:sausage}
\end{figure}
\begin{proposition}\label{prop:xyz}
	Let $\Sigma_g$ be a compact oriented genus $g\geq 2$ surface with a cellular decomposition as in Figure \ref{fig:xyz}. Let $G=D_n, n\geq 3$ or $A_4,S_4,A_5$ or $Q_8.$ Then 
	\begin{itemize}
	\item[-]There exists $x,y,z\in G$ non central elements such that $xyz=1_G,$ and $G=\langle x,y\rangle.$ 
	\item[-]If $x,y,z\in G$ satisfies those conditions, then the holonomy $\rho=\mathrm{hol}_c$ of the cocycle $c$ described in Figure \ref{fig:xyz} is compatible with the pants decomposition shown in the same figure.
	\item[-]If moreover $G\neq Q_8,$ then $\rho$ may be interpreted as a representation $\rho:\pi_1(\Sigma_g)\longrightarrow \so,$ and as such, it lifts to $\su.$
	\end{itemize}
\end{proposition}
\begin{proof}
It follows from the definition of $c$ that the value of its holonomy $\rho=\mathrm{hol}_c$ on any curve of the pants decomposition is $x,y$ or $z,$ up to conjugation and possibly inversion. Hence, any curve of the pants decomposition is mapped to a non central element. Moreover, the restriction of $\rho$ to any pants of the decomposition is conjugated to the map $F_2 \rightarrow G$ which maps the two generators of the free group $F_2$ to $x$ and $y.$ This and the hypothesis $G=\langle x,y \rangle$ implies that the restriction of $\rho$ to each pants of the decomposition is surjective onto $G,$ hence non abelian. This of course also implies that $\rho$ itself is surjective.

If $G$ is a subgroup of $\so,$ then the $\so$-cocycle $c$ may be lifted to $\su.$ Indeed, the top hexagons of the decomposition are all identical; to lift $c$ it suffices to choose lifts $\tilde{x},\tilde{y}$ and $\tilde{z}$ so that $\tilde{x}\tilde{y}\tilde{z}=1_{\su}.$ On the edges that belong to the bottom hexagons, we simply lift $I_3\in \so$ to $I_2\in \su.$ This defines indeed a $\su$ cocycle $\tilde{c}$ that lifts $c.$

Finally, we prove that in each case $G=D_n,n\geq 3$ or $A_4,S_4,A_5$ or $Q_8,$ a triple $x,y,z\in G$ of non central elements such that $xyz=1_G$ and $G=\langle x,y\rangle$ exists. 
\begin{itemize}
\item[-]For $G=D_n=\langle r,s | s^2=1, rs=sr^{-1}\rangle,$ one can take $x=s, y=r$ and $z=sr.$ These elements are non central when $n\geq 3.$
\item[-]For $G=A_4,$ one can take $x=(12)(34), y=(123)$ and $z=(234).$ Then $x$ and $y$ do not commute (in particular, $x,y$ and $z$ are non central), and one can see that $\langle x,y \rangle$ contains both the subgroup of all double transpositions and a $3$-cycle, hence its order is divisible by $3$ and $4,$ hence $\langle x,y \rangle= A_4.$
\item[-]For $G=S_4,$ one can take $x=(12), y=(1234)$ and $z=(324).$ Again, $x,y$ and $z$ are non central and $S_4=\langle x,y \rangle$ since it contains all transpositions $(i \ i+1)=y^{i-1}x y^{1-i}$ with $1\leq i\leq 4.$
\item[-]For $G=A_5,$ one can take $x=(12)(34), y=(12345), z=(254).$ These elements are obviously non central. We claim that $A_5=\langle x,y\rangle.$ Indeed, since $x,y,z$ have orders $2,5$ and $3,$ the subgroup $\langle x,y \rangle$ has order a multiple of $30,$ and index at most $2.$ But $A_5$ is simple, and a subgroup of index $2$ is always normal, since the index must be $1.$
\item[-]For $G=Q_8=\langle I,J,K,-1 | I^2=J^2=K^2=-1, (-1)^2=1, IJ=K \rangle,$ we can take $x=I, y=J$ and $z=-K,$ as these elements are non central and $Q_8$ is generated by $I$ and $J.$  
\end{itemize}
\end{proof}
Proposition \ref{prop:xyz}, together with Propostion \ref{prop:orbit_finite} and \ref{prop:orbit_quaternion} show that any finite non abelian image $\so$ representation of $\pi_1(\Sigma_g)$ that lift to $\su$ has a compatible pants decomposition. It remains to treat the case of representations that do not lift to $\su.$ We will make use the following lemma:
\begin{lemma}\label{lemma:non-lifting_elements} Let $x,y\in \so$ be commuting order $2$ elements such that $x\neq y^{\pm 1},$ and let $\tilde{x},\tilde{y}$ be any lifts of $x,y$ in $\su.$ Then $[\tilde{x},\tilde{y}]=-I_2.$
\end{lemma}
\begin{proof}
Elements of order $2$ in $\so$ lift to order $4$ elements in $\su$ since the only order $2$ elements in $\su$ are $\pm I_2.$ Moreover, two order $4$ elements in $\su$ must be co-diagonalizable with eigenvalues $\pm i,$ hence they must be equal or inverse of one another. Since $x\neq y^{\pm 1}$ one also has that $\tilde{x}\neq \tilde{y}^{-1}.$ So $\tilde{x}$ and $\tilde{y}$ are non commuting, therefore $[\tilde{x},\tilde{y}]=-I_2$ since their projections to $\so$ are commuting.
\end{proof}
\begin{proposition}\label{prop:xyz_nonlifting}
For $G=D_n, n\geq 3$ even, or $A_4,S_4$ or $A_5,$ there is a surjective representation $\rho:\pi_1(\Sigma_g)\longrightarrow \so$ with image isomorphic to $G,$ compatible with a square type pants decomposition and such that $\rho$ does not lift to $\su.$
\end{proposition}
 \begin{proof}
 Let $c$ be the cocycle defined in the proof of Proposition \ref{prop:xyz}. We will modify $c$ on the edges of a single hexagon of the cellular decomposition of Figure \ref{fig:xyz},
 and obtain another cocycle $c'$ whose holonomy is still surjective onto $G$ and compatible with the pants decomposition, but so that $\varepsilon(c')=-I_2.$ 
 
 By Proposition \ref{lemma:lift} this shows that $\rho'=\mathrm{hol}_{c'}$ does not lift. Take a top hexagon of the cellular decomposition, and replace the holonomy of edges: $1,x,1,y,1,z$ by $a,x,a^{-1},y,1,z$ where $a\neq x^{\pm 1}$ has order $2$ and commutes with $x.$ The holonomy of this hexagon is still $1.$ The corresponding bottom hexagon also still has holonomy $1,$ so this defines a $G$-cocycle $c'.$ We have not changed the holonomy of the pants decomposition curves, so they are still non central, the restriction to any pair of pants is still non abelian since $y$ and $z$ did not commute, and $\rho'$ is still surjective since its restriction to at least $1$ pair of pants of the decomposition coincides with the restriction of $\rho,$ which was surjective. 
 
 Notice that in each case covered in the proof of Proposition \ref{prop:xyz}, the element $x$ was chosen of order $2.$ Lemma \ref{lemma:non-lifting_elements} implies that $\varepsilon(c)=-\varepsilon(c'),$ hence $\rho'$ does not lift to $\su.$
 
 To conclude, it remains to show in each case that there is a choice of an element $a$ of order $2$ such that $a\neq x^{\pm 1}$ and $[a,x]=1.$ We list again the cases:
 \begin{itemize}
 \item[-]If $G=D_n,n\geq 3$ even, the element $a=r^{n/2}$ commutes with $x=s.$
 \item[-]If $G=A_4,$ the element $a=(13)(24)$ commutes with $x=(12)(34).$
 \item[-]If $G=S_4,$ the element $a=(34)$ commutes with $x=(12).$
 \item[-]If $G=A_5,$ the element $a=(13)(24)$ commutes with $x=(12)(34).$
\end{itemize}   \end{proof}
\begin{proposition}\label{prop:sausage}
	Let $\Sigma_g$ be a compact oriented genus $g\geq 2$ surface with a cellular decomposition as in Figure \ref{fig:sausage}. Let $G=D_n, n\geq 3$ or $A_4,S_4$ or $A_5.$ Then 
	\begin{itemize}
	\item[-]There exists $x,y,a\in G,$ with $x,y$ non commuting elements  and such that $[x,a]y=1_G,$ and $G=\langle x,a\rangle.$ 
	\item[-]If $x,y,a\in G$ satisfies those conditions, then the holonomy $\rho=\mathrm{hol}_c$ of the cocycle $c$ described in Figure \ref{fig:sausage} is compatible with the pants decomposition shown in the same figure.
	\item[-]The representation $\rho,$ interpreted as a representation $\rho:\pi_1(\Sigma_g)\longrightarrow \so$ lifts to $\su.$
	\end{itemize}
\end{proposition}
\begin{proof}
We again start by proving the last two points, which are easier: notice that the holonomy of the top hexagons of the cellular decomposition are all $[x,a]y,$ while the holonomy of the bottom hexagons are all $aa^{-1}=1_G.$ Thus $c$ defines a $G$-cocycle if and only if $[x,a]y=1_G.$ The representation $\rho=\mathrm{hol}_c$ is surjective since $G=\langle a,x \rangle$ and the image of $\rho$ contains $a$ and $x,$ and it is irreducible in restriction to each pants since $x$ and $y$ do not commute. 

For the lifting property, notice by one can choose lifts $\tilde{x},\tilde{a},\tilde{y}\in \su$ of $x,y,a\in \so$ so that $[\tilde{x},\tilde{a}]\tilde{y}=I_2,$ since $\tilde{y}$ appears with odd power in the product $[\tilde{x},\tilde{a}]\tilde{y}.$ Then using the same lifts for each hexagon it is clear that we get $\su$-cocycle $\tilde{c}$ that is a lift of $c.$

We now produce elements $x,a,y$ for each choice of finite group $G$ listed in the proposition:
\begin{itemize}
\item[-]If $G=D_n, n\geq 3$ odd, take $x=s, a=r^{(n+1)/2}$ and $y=[a,x]=r^{n+1}=r.$ Then $x$ and $y$ are non commuting and generate $G$ (thus $a$ and $x$ also generate $G.$)
\item[-]If $G=D_n, n\geq 3$ even, take $x=s, a=r$ and $y=[a,x]=r^2.$ Since $n\geq 3,$ $x$ and $y$ are non commuting and moreover $G=\langle a,x\rangle.$
\item[-]If $G=A_4,$ take $x=(123),$ $a=(234)$ and $y=[a,x]=(14)(23).$ The elements $x$ and $y$ are non commuting. We claim that $\langle a,x\rangle=G:$ indeed, $\langle a,x \rangle$ contains the $3$-cycle $x$, and contains a double transposition $y.$ It actually contains the subgroup generated by double transpositions since it also contains $xyx^{-1}=(13)(24).$ Thus its order is divisible both by $3$ and $4,$ therefore it is $A_4.$
\item[-]If $G=S_4,$ take $x=(1234), a=(23)$ and $y=[a,x]=(234).$ It is again clear that $x$ and $y$ are non commuting. The subgroup $\langle a,x \rangle$ contains a $3$ and a $4$ cycle, hence its index is $1$ or $2.$ But the only index $2$ subgroup of any $S_n$ is $A_n,$ and $\langle a,x \rangle$ is not included in $A_4$ since $x\notin A_4,$ so we must have $S_4=\langle x,a\rangle.$
\item[-]If $G=A_5,$ take $x=(13)(24), a=(345)$ and $y=[a,x]=(13452).$ Then $x$ and $y$ are non commuting, and $\langle a,x \rangle$ has order divisible by $2,3$ and $5.$ Thus it has index at most $2$ in $A_5$ therefore by simplicity of $A_5$ one has $A_5=\langle a,x\rangle.$
\end{itemize}
\end{proof}
Again we complement the previous proposition with a proposition with deals with non abelian finite image $\so$ representations that do not lift to $\su:$
\begin{proposition}\label{prop:sausage_nonlifting}
For $G=D_n, n\geq 3$ even, or $A_4,S_4$ or $A_5,$ there is a surjective representation $\rho:\pi_1(\Sigma_g)\longrightarrow \so$ with image isomorphic to $G,$ compatible with a sausage type pants decomposition and such that $\rho$ does not lift to $\su.$
\end{proposition}
\begin{proof}
We follow the same strategy as for Proposition \ref{prop:xyz_nonlifting}, modifying the cocycle $c$ of Proposition \ref{prop:sausage} on the edges of a single top hexagon. This time the top hexagon will be the leftmost hexagon in Figure \ref{fig:sausage}. The new holonomy will still be surjective as the restriction to the one-holed torus on the right side of the decomposition will not have changed, and furthermore it will be clear that it is still compatible with the pants decomposition. The sequence of holonomy of edges for the distinguished top hexagon was initially: $x,a,x^{-1},a^{-1},y,1,$ we change it to the following:
\begin{itemize}
\item[-]If $G=D_n,n\geq 3$ even, to $x,ab, x^{-1},b^{-1}a^{-1},y,1$ where $b=r^{n/2}$ commutes with $x=s.$
\item[-]If $G=A_4,$ to $x,a,x^{-1},a^{-1}b,y,b^{-1}$ where $b=(12)(34)$ commutes with $y=(14)(23).$
\item[-]If $G=S_4,$ to $xb,a,b^{-1}x^{-1},a^{-1},y,1$ where $b=(14)$ commutes with $a=(23).$
\item[-]If $G=A_5,$ to $x,ab,x^{-1},b^{-1}a^{-1},y,1$ where $b=(12)(34)$ commutes with $x=(13)(24).$
\end{itemize}
Lemma \ref{lemma:non-lifting_elements} then implies in each case that $\varepsilon(c')=-\varepsilon(c),$ hence the new holonomy does not lift.
\end{proof}
\bibliographystyle{hamsalpha}
\bibliography{biblio}
\end{document}